\documentclass{amsart}

\usepackage{amsmath, amssymb, xypic}

\newcommand{\frc}{\mathfrak c}

\newcommand{\bfB}{{\mathbf B}}
\newcommand{\bbN}{{\mathbb N}}
\newcommand{\bbC}{\mathbb C}

\newtheorem{thm}{Theorem}

\newtheorem{lemma}[thm]{Lemma}
\newtheorem{prop}[thm]{Proposition}

\theoremstyle{definition}

\newcounter{my_enumerate_counter}
\newcommand{\pushcounter}{\setcounter{my_enumerate_counter}{\value{enumi}}}
\newcommand{\popcounter}{\setcounter{enumi}{\value{my_enumerate_counter}}}

\newcommand{\FileName}[1]{\thanks{Filename: {\tt #1}}}

\title{Orthonormal bases of Hilbert spaces}

\author{Ilijas Farah}

\address{Department of Mathematics and Statistics\\
York University\\
4700 Keele Street\\
North York, Ontario\\ Canada, M3J
1P3\\
and Matematicki Institut, Kneza Mihaila 34, Belgrade, Serbia}

\urladdr{http://www.math.yorku.ca/$\sim$ifarah}

\email{ifarah@mathstat.yorku.ca}
\thanks{Partially supported by NSERC}

\FileName{2008l09-basis.tex}

\date{\today}
\begin{document}
\maketitle

Assume $H$ is a Hilbert space and $K$ is a dense linear (not
necessarily closed) subspace. The question whether $K$ necessarily
contains an orthonormal basis for $H$ even when $H$ is nonseparable
was mentioned by Bruce Blackadar in an informal conversation during
the Canadian Mathematical Society meeting in Ottawa in December 2008
and this note provides a negative answer. Note that the Gram--Schmidt
process gives a positive answer when $H$ is separable.

 I will use $\aleph_1$ to denote both the
first uncountable ordinal and the first uncountable cardinal and I
will use $\frc=2^{\aleph_0}$ to denote both the cardinality of the
continuum and the least ordinal of this cardinality. All bases are
orthonormal.

For cardinals $\lambda<\theta$ 
consider   $\ell^2(\lambda)$ as a
subspace of $\ell^2(\theta)$ consisting of vectors supported on the first
$\lambda$ coordinates. Let $p_\lambda$ denote the projection of $\ell^2(\theta)$
to $\ell^2(\lambda)$.

\begin{lemma} \label{L1} Assume $\lambda<\theta$ are infinite cardinals 
such that $\theta$ is regular and $x_\gamma$, for
$\gamma<\theta$, is an orthonormal family in $\ell^2(\theta)$. Then there is
$\gamma_0<\theta$ such that $x_\gamma$ is orthogonal to $\ell^2(\lambda)$ for
all $\gamma\geq \gamma_0$.
\end{lemma}

\begin{proof}For $\alpha\leq\theta$ let $X(\alpha)$ denote the closed linear span of $x_\gamma$ for $\gamma<\alpha$.
Let $e_\xi$, for $\xi<\lambda $, be the standard basis for $\ell^2(\lambda)$. Let
$\alpha(\xi)<\kappa$ be the minimal ordinal such that the
projection of $e_\xi$ to $X(\theta)$ is in $X(\alpha(\xi))$. Since $\theta>\lambda$ we have 
$\alpha(\xi)<\theta$ and by the regularity of $\theta$ we  have that  
$\gamma_0=\sup_{\xi<\lambda} \alpha(\xi)<\theta$ is as required.
\end{proof}

\begin{lemma}\label{L3} Assume $\lambda<\theta$ are infinite cardinals 
such that $\theta$ is regular and $\lambda^{\aleph_0}\geq \theta$. 
Then there is a dense
linear subspace $K$ of $\ell^2(\theta)$ such that the kernel of the
restriction of $p_\lambda$ to $K$ is $\{0\}$. Such $K$ does not contain an
 orthonormal family of size greater than $\lambda$.
\end{lemma}

\begin{proof}
 Let $z_\gamma$, for $\gamma<\theta$, be
a dense subset of $\ell^2(\theta)$. We shall find $y_{\gamma,m}$, for
$\gamma<\theta$ and $m\in \bbN$, such that
$\|y_{\gamma,m}-z_\gamma\|\leq 1/m$ for all $\gamma$ and $m$ and
$p_\lambda(y_{\gamma,m})$, for $\gamma<\theta$ and $m\in \bbN$, are
linearly independent.

Fix a Hamel basis $\bfB$ for $\ell^2(\lambda)$ considered as a vector space over
$\bbC$. We have that $|\bfB|=\lambda^{\aleph_0}\geq \theta$.  
Assume $y_{\gamma,m}$ have been constructed for all
$\gamma<\alpha$ and all $m$. Let $F$ be the minimal subset of $\bfB$
such that $\{p_\lambda(z_\alpha)\}\cup \{y_{\gamma,m}: \gamma<\alpha, m\in
\bbN\}$ is included in the linear span of $F$. Then $|F|\leq
|\gamma|+\aleph_0<\theta\leq |\bfB|$. Fix distinct
vectors $t_m$, for $m\in \bbN$, in $\bfB\setminus F$ and let
$y_{\alpha,m}=z_\alpha+ \frac 1m  t_m$. (We are assuming  $t_m$ are
unit vectors, but this is not required from $y_{\alpha,m}$.) Then
$\|y_{\alpha,m}-z_\alpha\|=\frac 1m$ and and  $y_{\gamma,m}$, for
$\gamma\leq \alpha$ and $m\in \bbN$, are linearly independent.

This describes the recursive construction. The linear span  $K$ of
$\{y_{\gamma,m}:\gamma<\theta,m\in \bbN\}$ is dense and for $x\in K$
we have $p_\lambda(x)=0$ if and only if $x=0$.
Lemma~\ref{L1} implies that $K$ cannot contain an
orthonormal family of size greater than $\lambda$.
\end{proof}

\begin{prop} \label{P4} Every nonseparable Hilbert space $H$ contains a dense
subspace that contains no basis for $H$.
\end{prop}

\begin{proof}
We may assume $H=\ell^2(\theta)$ for some uncountable cardinal
$\theta$. In the case when $\theta\leq 2^{\aleph_0}$
 the existence of  $K$ is guaranteed by the case $\lambda=\aleph_0$ of 
 Lemma~\ref{L3}.

 We may therefore assume $\theta>2^{\aleph_0}$
and write  $H=\ell^2(\frc)\oplus \ell^2(\theta)$. Let $H_0$ be a separable subspace of
$\ell^2(\frc)$ and let $K$ be a dense subspace of $\ell^2(\frc)$ as in
Lemma~\ref{L3}, so that the projection $p_0$ of $\ell^2(\frc)$ to $H_0$
satisfies $\ker(p_0)\cap K=\{0\}$.

The dense subspace  $K_1=K\oplus \ell^2(\theta)$ of $H$ contains no
basis for $H$. Assume the contrary and let $\eta_\gamma$, for
$\gamma<\theta$, be such a basis. Write $q_0$ for the projection of
$H$ to $H_0$ and $q_{\frc}$ for the projection of $H$ to $\ell^2(\frc)$.
By Lemma~\ref{L1} the set $X=\{\gamma: q_{0}(\eta_\gamma)\neq 0\}$ is
countable. On the other hand, since the vectors
$\{q_{\frc}(\eta_\gamma): \gamma<\theta\}$ span $\ell^2(\frc)$ the set
$\{\gamma<\theta: q_{\frc}(\eta_\gamma)\neq 0\}$ is uncountable.
Therefore for some $\gamma$ we have $q_0(\eta_\gamma)=0$ and
$q_{{\frc}}(\eta_\gamma)\neq 0$. Since $p_0
(q_{{\frc}}(\eta_\gamma))=q_0(\eta_\gamma)$ this contradicts the
choice of $K$.
\end{proof}

I shall end by providing an explanation why the subspace $K$ of
$\ell^2(\theta)$ constructed in the proof of Proposition~\ref{P4} has
a much stronger property when $\theta\leq 2^{\aleph_0}$ than when, for
example, $\theta=(2^{\aleph_0})^+$.

\begin{prop} Assume $\theta$ is a regular cardinal. The following are equivalent.
\begin{enumerate}
\item  \label{P3.1} For all cardinals $\lambda<\theta$ we have
$\lambda^{\aleph_0}<\theta$.

\item \label{P3.2} If  $Y$ is a linear subspace of some Hilbert space such that
$|Y|=\theta$ then $Y$ contains an orthonormal family of size
$\theta$. \end{enumerate}
\end{prop}

\begin{proof} By Lemma~\ref{L3}, \eqref{P3.2} implies
\eqref{P3.1}.
 Now we assume  \eqref{P3.1} and prove  \eqref{P3.2}. We may assume
$Y$ is a subspace of $\ell^2(\theta)$. Let $y_\gamma$,
$\gamma<\theta$, be distinct vectors in $Y$. For each $\gamma$ let
$X_\gamma$ be the support of $y_\gamma$. Applying the generalized
$\Delta$-system lemma (\cite[Theorem~1.6]{Ku:Book}, with
$\kappa=\aleph_1$) to $X_\gamma$, for $\gamma<\theta$,  we find
$X\subseteq \theta$ and $I_1\subseteq \theta$ of cardinality $\theta$
such that $X_\beta\cap X_\gamma=X$ for all $\beta\neq \gamma$ in
$I_1$. Let $p$ denote the projection of $\ell^2(\theta)$ to
$\ell^2(X)$. Since $X$ is at most countable and $\theta>2^{\aleph_0}$
is regular we can find $y\in \ell^2(X)$ and $I_2\subseteq I_1$ of
cardinality $\theta$ such that $p(y_\gamma)=y$ for all $\gamma\in
I_2$. Then $z_\gamma=y_\gamma-y$ for $\gamma\in I_2$ clearly form an
orthonormal family of size $\theta$.
\end{proof}

I would like to thank Justin Moore for pointing out to a few typos. 
\providecommand{\bysame}{\leavevmode\hbox to3em{\hrulefill}\thinspace}
\providecommand{\MR}{\relax\ifhmode\unskip\space\fi MR }
\providecommand{\MRhref}[2]{%
  \href{http://www.ams.org/mathscinet-getitem?mr=#1}{#2}
}
\providecommand{\href}[2]{#2}

\end{document}